\documentclass[a4paper,10pt]{amsart}

\usepackage{mathtools,amssymb,amsfonts,stmaryrd,amsmath,amsthm,enumerate,hyperref,color}
\usepackage[utf8]{inputenc}
\usepackage{enumitem}

\usepackage[all,cmtip]{xy}     
\usepackage{wrapfig,tikz,tikz-cd}
\usetikzlibrary{arrows,decorations.pathreplacing}
\usepackage{caption}
\usepackage{comment}

\usepackage[capitalise]{cleveref}

\usepackage[parfill]{parskip}
\usepackage[margin=1.2in]{geometry}
\AtBeginDocument{\addtocontents{toc}{\protect\setlength{\parskip}{0pt}}}

\newtheorem{theorem}{Theorem}[section]
\newtheorem{lemma}[theorem]{Lemma}
\newtheorem{proposition}[theorem]{Proposition}
\newtheorem{corollary}[theorem]{Corollary}

\theoremstyle{definition}
\newtheorem{definition}[theorem]{Definition}
\newtheorem{remark}[theorem]{Remark}
\newtheorem{example}[theorem]{Example}

\theoremstyle{remark}

\newcommand{\Z}{\mathbb{Z}}
\newcommand{\R}{\mathbb{R}}
\newcommand{\F}{\mathbb{F}}
\DeclareMathOperator{\rk}{rk}
\DeclareMathOperator{\lk}{lk}

\begin{document}

\title[Equivariant formality and subgroups of right-angled Coxeter groups]{Equivariant formality and the cohomology of subgroups of right-angled Coxeter groups} 

\date{}

\author[S.~Amelotte]{Steven Amelotte}
\address{School of Mathematics and Statistics, 4302 Herzberg Laboratories, Carleton University, Ottawa, ON, K1S 5B6, Canada}
\email{steven.amelotte@carleton.ca} 

\author[V.~Gorchakov]{Vladimir Gorchakov}
\address{Department of Mathematics, University of Western Ontario, London, ON, N6A 5B7, Canada
}
\email{vyugorchakov@gmail.com}

\keywords{Right-angled Coxeter groups, real moment-angle complexes, equivariant formality, equivariant cohomology}
\subjclass[2020]{20F55, 20J06, 55N91, 55R35, 57S12, 57S17, 57M07}
% 20J06 = Cohomology of groups
% 57M07 = Topological methods in group theory
% 57S17 = Finite transformation groups
% 55R35 = Classifying spaces of groups and $H$-spaces

\begin{abstract}
We construct models for the classifying spaces of coabelian subgroups of right-angled Coxeter groups as homotopy orbit spaces of real moment-angle complexes, generalizing well-known models for the classifying space of a right-angled Coxeter group and its commutator subgroup. This identifies the cohomology of these groups with the Borel equivariant cohomology of elementary abelian $2$-group actions on cubical subcomplexes of a cube $[-1,1]^m$. We then characterize equivariant formality for these actions, leading to a simple graph-theoretic criterion for when the cohomology of a coabelian subgroup is free as a module over the cohomology of the quotient by the commutator subgroup of the right-angled Coxeter group.
\end{abstract}

\maketitle

\section{Introduction}

Given a simple graph $\Gamma$ on the vertex set $[m]=\{1,\ldots,m\}$ with edge set $E(\Gamma)$, the \emph{right-angled Coxeter group} $W_\Gamma$ associated to $\Gamma$ is a classical object of study arising in combinatorial and geometric group theory defined by the presentation
\[
W_\Gamma= \left\langle g_1,\ldots,g_m \,:\, g_i^2=1 \text{ for all } i\in [m] \text{ and } [g_i,g_j]=1 \text{ for } \{i,j\}\in E(\Gamma) \right\rangle.
\]
Unlike the right-angled Artin group $A_\Gamma$, which admits a finite classifying space in the form of a subcomplex of the $m$-torus $T^m$ by a result of Charney and Davis~\cite{CD}, $W_\Gamma$ contains torsion, and hence is not of finite cohomological dimension. It is, however, of finite virtual cohomological dimension (meaning it contains a finite-index subgroup of finite cohomological dimension), as witnessed by the commutator subgroup $[W_\Gamma,W_\Gamma]$, which admits a finite classifying space in the form of a subcomplex of the $m$-cube $[-1,1]^m$.

More precisely, classifying spaces of each of the groups above can be obtained as special cases of a common construction at the foundations of toric topology known as the polyhedral product functor. For each pair of spaces $(X,A)$ and simplicial complex~$K$ on~$m$ vertices, this functor carves out a topological space $(X,A)^K$ from the cartesian product $X^m$ (see \cref{subsec_realMACs}). In particular, the \emph{real moment-angle complex} $\R\mathcal{Z}_K=(D^1,S^0)^K$ and the \emph{real Davis--Januszkiewicz space} $\R DJ_K=(\R P^\infty,\, pt)^K$ are well-known models for the classifying spaces of~$[W_\Gamma,W_\Gamma]$ and~$W_\Gamma$, respectively, when~$K$ is the flag simplicial complex with underlying graph~$K^{(1)}=\Gamma$ (see~\cite{PV} or \cref{thm_aspherical} below); here $D^1=[-1,1]$ and $S^0=\partial D^1$.

By definition, $\R\mathcal{Z}_K$ is a cubical subcomplex of the cube $(D^1)^m$, from which it inherits a canonical action of the elementary abelian $2$-group $\Z_2^m$ by coordinatewise reflection. The homotopy orbits (i.e., Borel construction) of this $\Z_2^m$-space is homotopy equivalent to the real Davis--Januszkiewicz space: $(\R\mathcal{Z}_K)_{\Z_2^m} = E\Z_2^m \times_{\Z_2^m} \R\mathcal{Z}_K \simeq \R DJ_K$. Thus, while the right-angled Coxeter group $W_\Gamma$ does not admit a finite classifying space, it does admit one which is the homotopy quotient of a finite $\Z_2^m$-invariant subcomplex of $(D^1)^m$. In this paper, we observe that the same is true of all coabelian subgroups $G$ of $W_\Gamma$ (\cref{prop_corresp}) and use this to study the cohomology of these groups as modules over the polynomial ring given by the cohomology of the quotient $G/[W_\Gamma,W_\Gamma]$.

Note that if $G$ is a subgroup of $W_\Gamma$ containing the commutator subgroup, then $G/[W_\Gamma,W_\Gamma]$ is a subgroup of the abelianization $W_\Gamma/[W_\Gamma,W_\Gamma]=\Z_2^m$. If $K$ is a simplicial complex with underlying graph $\Gamma$, then by considering restricted actions of subgroups of $\Z_2^m$ on the real moment-angle complex $\R\mathcal{Z}_K$, we establish a bijection
\begin{align*}
\Bigl\{ \substack{\text{homotopy quotients of}\\ \R\mathcal{Z}_K\text{by subgroups of }\Z_2^m} \Bigr\} \;&\longleftrightarrow\; \Bigl\{ \substack{\text{coabelian subgroups}\\ \text{of the RACG }W_\Gamma} \Bigr\} \\
\text{\small$(\R\mathcal{Z}_K)_A$} \quad&\longmapsto\quad \text{\small$\pi_1\bigl( (\R\mathcal{Z}_K)_A \bigr)$} \\
\text{\small$(\R\mathcal{Z}_K)_{G/[W_\Gamma,W_\Gamma]}$} \;\;&\longmapsfrom\quad \text{\small$G$}
\end{align*}
where the right-to-left assignment can be identified with the classifying space functor when $K$ is flag. Here $\pi_1\bigl( (\R\mathcal{Z}_K)_A \bigr)$ can be viewed as a subgroup of $W_\Gamma$ via the injection of fundamental groups induced by a natural fibre inclusion $(\R\mathcal{Z}_K)_A \to \R DJ_K$ (see~\eqref{two_diagrams} or the proof of \cref{prop_corresp}). In case the subgroup $A$ of $\Z_2^m$ acts freely on the real moment-angle complex, $\R\mathcal{Z}_K/A \simeq (\R\mathcal{Z}_K)_A$ is called a \emph{real toric space}. These quotients include all real toric manifolds (i.e., real loci of nonsingular complete toric varieties) and small covers, more generally~\cite{DJ}, and correspond under the bijection above to subgroups of $W_\Gamma$ of finite cohomological dimension.

Our main result is the following combinatorial classification of coabelian subgroups $G$ of $W_\Gamma$ for which $\mathrm{H}^*(G;\F_2)$ is free over the polynomial ring $\mathrm{H}^*(G/[W_\Gamma,W_\Gamma];\F_2)$.

\begin{theorem} \label{main_thm_1}
Let $\Gamma$ be a simple graph on the vertex set $[m]$, and let $G$ be a coabelian subgroup of $W_\Gamma$. Let $I\subseteq [m]$ be the minimal subset for which $\Z_2^I$ contains the subgroup $G/[W_\Gamma,W_\Gamma]$ of $W_\Gamma/[W_\Gamma,W_\Gamma] = \Z_2^m$ (see \cref{def_hull}). Then the following conditions are equivalent:
\begin{enumerate}[label={\normalfont(\alph*)}]
\item\label{main_thm_1a} $\mathrm{H}^*(G;\F_2)$ is a free module over the polynomial ring $\mathrm{H}^*(G/[W_\Gamma,W_\Gamma];\F_2)$;
\item\label{main_thm_1b} $\mathrm{H}^*(G^\rtimes;\F_2)$ is a free module over $\mathrm{H}^*(G^\rtimes/[W_\Gamma,W_\Gamma];\F_2) \cong \F_2[x_i : i\in I]$, where
\[ G^\rtimes = \overline{W}_{\Gamma_I} \rtimes [W_{\Gamma_J},W_{\Gamma_J}], \quad J=[m]\setminus I; \]
\item\label{main_thm_1c} the induced subgraph $\Gamma_I$ is complete and $\{i,j_1\}, \{i,j_2\} \in E(\Gamma)$ for every non-edge $\{j_1,j_2\}$ of $\Gamma$ and every $i\in I\setminus\{j_1,j_2\}$.
\end{enumerate}
\end{theorem}

Here $\overline{W}_{\Gamma_I}$ denotes the normal closure of the standard parabolic subgroup $W_{\Gamma_I}$ of $W_\Gamma$; classifying spaces of these (non-right-angled) Coxeter groups can be described in terms of real Davis--Januszkiewicz spaces (\cref{lem_coord_quot}). The equivalent conditions above imply the collapse of the Lyndon--Hochschild--Serre spectral sequence associated to the short exact sequence of groups 
\[ 1\to [W_\Gamma,W_\Gamma] \to G \to G/[W_\Gamma,W_\Gamma] \to 1, \]
which can be identified with the Serre spectral sequence of the Borel fibration
\[ \R\mathcal{Z}_K\to (\R\mathcal{Z}_K)_A\to BA \]
when $K$ is flag and $A=G/[W_\Gamma,W_\Gamma]$, so that $(\R\mathcal{Z}_K)_A\simeq BG$. The action of a subgroup $A\leqslant\Z_2^m$ on $\R\mathcal{Z}_K$ is called \emph{equivariantly formal} when the equivariant cohomology $H_A^*(\R\mathcal{Z}_K;\F_2)$ is free as a module over the polynomial ring $H^*(BA;\F_2)$, or equivalently, when $A$ acts trivially on $H^*(\R\mathcal{Z}_K;\F_2)$ and the Serre spectral sequence of the fibration above degenerates at the second page (see \cref{thm_bettisumcrit}). The analogous problem of classifying equivariantly formal torus actions on ordinary moment-angle complexes $\mathcal{Z}_K=(D^2,S^1)^K$ was solved by the first author and Briggs in~\cite{AB}, and here we obtain \cref{main_thm_1} by proving the following real analogue of~\cite[Theorem~C]{AB}.

\begin{theorem} \label{main_thm_2}
Let $K$ be a flag complex on the vertex set $[m]$, and let $A$ be a subgroup of $\Z_2^m$. Let $I\subseteq [m]$ be the minimal subset for which the coordinate subgroup $\Z_2^I$ contains $A$ (see \cref{def_hull}). Then the following conditions are equivalent:
\begin{enumerate}[label={\normalfont(\alph*)}]
\item\label{main_thm_2a} $\R\mathcal{Z}_K$ is an equivariantly formal $A$-space over $\F_2$;
\item\label{main_thm_2b} $\R\mathcal{Z}_K$ is an equivariantly formal $\Z_2^I$-space over $\F_2$;
\item\label{main_thm_2c} $I\in K$ and $\{i,j_1\}, \{i,j_2\} \in K$ for every missing edge $\{j_1,j_2\} \notin K$ and every $i\in I\setminus\{j_1,j_2\}$.
\end{enumerate}
\end{theorem}

We remark that~\ref{main_thm_2c} above is equivalent to the condition that $I\in K$ and $\Z_2^I$ acts trivially on $H^1(\R\mathcal{Z}_K;k) \cong \mathrm{H}^1([W_\Gamma,W_\Gamma];k)$. (This can be seen by describing the $\Z_2^m$-action in terms of a Hochster formula description of the cohomology of $\R\mathcal{Z}_K$, exactly as in~\cite[Lemma~3.5]{AB}.)
In the non-flag case, $\Z_2^I$ acting trivially on $H^1(\R\mathcal{Z}_K;k)$ does not imply the action is trivial on $H^*(\R\mathcal{Z}_K;k)$, and the latter may depend on the characteristic of the field $k$, making a purely combinatorial characterization impossible. See \cref{prop_ef_hull} and \cref{thm_ef} for a characterization of equivariant formality for real moment-angle complexes when $K$ is not necessarily flag.

The $\Z_2^m$-space $\R\mathcal{Z}_K$ will be our main tool throughout the paper, and consequently, as in~\cite{D}, we work primarily with the (equivariant) cohomology of spaces rather than the cohomology of groups,\footnote{To help avoid confusion, we use the notation $H^*(-;k)$ and $\mathrm{H}^*(-;k)$ to distinguish between the cohomology of spaces and groups, respectively.} and with simplicial complexes rather than graphs. By specializing to the case of a flag simplicial complex $K$, we recover the cohomology of groups associated to the underlying graph $\Gamma=K^{(1)}$ given by the $1$-skeleton of $K$. Other tools used heavily throughout include a criterion for equivariant formality in terms of a generalized Smith--Thom inequality $\dim_{\mathbb{F}_2}H^*(X^G;\mathbb{F}_2) \leqslant \dim_{\mathbb{F}_2}H^*(X;\mathbb{F}_2)$ for a $G$-space $X$ with fixed points $X^G$ (\cref{thm_bettisumcrit}), and a result of some independent interest relating the equivariant formality of $T^r$- and $\Z_2^r$-actions coming from a $T^r$-space with a compatible involution (\cref{thm_compatiblecrit}). Both are derived from the localization theorem in an appendix.

The paper is organized as follows. In \cref{sec_2} we review the construction of real moment-angle complexes and use them to construct models for classifying spaces of coabelian subgroups of RACGs. In particular, the semidirect product groups appearing in \cref{main_thm_1}\ref{main_thm_1b} are shown to admit classifying spaces given by homotopy quotients by the coordinate subgroups $\Z_2^I$ of $\Z_2^m$. In \cref{sec_3} we classify equivariantly formal actions of elementary abelian $2$-groups on $\R\mathcal{Z}_K$, proving \cref{main_thm_2} by first reducing the problem to the case of coordinate subgroup actions and then extending a combinatorial criterion from~\cite{AB} to our setting. In \cref{sec_4} we derive \cref{main_thm_1} from \cref{main_thm_2}. We delay until \cref{app_A} a number of general results on equivariant formality over~$\F_2$ which we use heavily throughout the paper or could not find in the literature.

\subsection*{Acknowledgements}
The authors are thankful to Matthias Franz for helpful discussions regarding the localization theorem (\cref{thm_localization}). The first author was partially supported by the Fields Institute for Research in Mathematical Sciences.

\section{Classifying spaces of subgroups of RACGs via real moment-angle complexes} \label{sec_2}

In this section we use a fundamental construction from toric topology to describe models for the classifying spaces of coabelian subgroups of right-angled Coxeter groups. Generalizing well-known models for the classifying spaces of right-angled Coxeter groups and their commutator subgroups, this identifies the group cohomology of these subgroups with the equivariant cohomology of certain elementary abelian $2$-group actions on cubical subcomplexes of a cube $[-1,1]^m$.

\subsection{Real moment-angle complexes} \label{subsec_realMACs}

Let $K$ be an abstract simplicial complex on the vertex set $[m]=\{1,\ldots,m\}$, and let $(\boldsymbol{X},\boldsymbol{A})=\left\{ (X_1,A_1),\ldots,(X_m,A_m) \right\}$ be a sequence of pairs of topological spaces $A_i\subseteq X_i$. For any $\sigma\subseteq [m]$, we write $(\boldsymbol{X},\boldsymbol{A})^\sigma$ for the subspace of the cartesian product $\prod_{i=1}^m X_i$ given by
\[
(\boldsymbol{X},\boldsymbol{A})^\sigma=\big\{ (x_1,\ldots,x_m)\in \textstyle\prod_{i=1}^mX_i \,:\, x_i \in A_i \text{ for } i\notin\sigma \big\}.
\]
The \emph{polyhedral product} $(\boldsymbol{X},\boldsymbol{A})^K$ is then the subspace of $\prod_{i=1}^m X_i$ defined by the colimit
\[
(\boldsymbol{X},\boldsymbol{A})^K=\operatorname*{colim}_{\sigma\in K}\, (\boldsymbol{X},\boldsymbol{A})^\sigma = \bigcup_{\sigma\in K}(\boldsymbol{X},\boldsymbol{A})^\sigma.
\]
When $(X_i,A_i)=(X,A)$ for all $i\in[m]$, we denote $(\boldsymbol{X},\boldsymbol{A})^K$ simply by $(X,A)^K$.

Key examples of polyhedral products in toric topology include the \emph{moment-angle complex} $\mathcal{Z}_K=(D^2,S^1)^K$ and the \emph{Davis--Januszkiewicz space} $DJ_K=(\mathbb{C}P^\infty,\, pt)^K$. (Here $D^2=\{z\in\mathbb{C}:|z|\leqslant 1\}$ and $S^1=\partial D^2$.) These spaces fit in a homotopy fibration
\[
\mathcal{Z}_K \longrightarrow DJ_K \longrightarrow BT^m,
\]
which can be identified up to homotopy with the Borel fibration of the natural $T^m=(S^1)^m$-action on $\mathcal{Z}_K\subseteq (D^2)^m$ defined by coordinatewise multiplication; see~\cite[Theorem~4.3.2]{BP}.

Real analogues of these spaces are given by the \emph{real moment-angle complex} $\R\mathcal{Z}_K=(D^1,S^0)^K$ and the \emph{real Davis--Januszkiewicz space} $\R DJ_K=(\R P^\infty,\, pt)^K$. (Here $D^1=[-1,1]$ and $S^0=\partial D^1$.) By definition, $\R\mathcal{Z}_K$ is a cubical subcomplex of the $m$-dimensional cube $(D^1)^m$, and the action of the elementary abelian $2$-group $\Z_2^m$ on the cube by coordinatewise reflection restricts to an action on~$\R\mathcal{Z}_K$. As above, these spaces fit in a homotopy fibration
\begin{equation} \label{main_fib}
\R\mathcal{Z}_K \longrightarrow \R DJ_K \longrightarrow B\Z_2^m,
\end{equation}
which is a polyhedral product model for the Borel fibration of the $\Z_2^m$-action on $\R\mathcal{Z}_K$.

Thus for every simplicial complex $K$ on $m$ vertices, the real moment-angle complex $\R\mathcal{Z}_K$ is naturally a $\Z_2^m$-space. Given any subgroup $A$ of $\Z_2^m$, the \emph{homotopy quotient} (or \emph{homotopy orbit space}) $(\R\mathcal{Z}_K)_A$ will always refer to the Borel construction
\[ 
(\R\mathcal{Z}_K)_A \coloneqq EA \times_A \R\mathcal{Z}_K 
\] 
of the restricted action of $A$ on $\R\mathcal{Z}_K$, which is the total space of the Borel fibration $\R\mathcal{Z}_K \to (\R\mathcal{Z}_K)_A \to BA$. 

Since any subgroup $A$ of $\Z_2^m$ is itself an elementary abelian $2$-group, we have $A\cong \Z_2^r$ for some $r\leqslant m$, and the equivariant cohomology $H_A^*(\R\mathcal{Z}_K;\F_2) \coloneqq H^*((\R\mathcal{Z}_K)_A;\F_2)$ is a graded module over the polynomial ring $H^*(BA;\F_2)\cong \F_2[x_1,\ldots,x_r]$, with $|x_i|=1$, via the map induced by the projection $(\R\mathcal{Z}_K)_A \to BA$.

The following summarizes some well-known results on the fundamental group and asphericity of real moment-angle complexes and real Davis--Januszkiewicz spaces. Recall that a simplicial complex $K$ is called a \emph{flag complex} if every set of pairwise connected vertices spans a face of $K$. Equivalently, $K$ is flag if it is the clique complex of the underlying graph $\Gamma=K^{(1)}$ given by the $1$-skeleton of $K$. 

\begin{theorem}[{\cite[Corollary~3.4]{PV}}; see also \cite{DJ},\cite{PRV}] \label{thm_aspherical}
Let $K$ be a simplicial complex with underlying graph $\Gamma$.
\begin{enumerate}[label={\normalfont(\alph*)}]
\item $\pi_1(\mathbb{R}DJ_K) \cong W_\Gamma$. \label{thm_aspherical_a}
\item $\pi_1(\mathbb{R}\mathcal{Z}_K) \cong [W_\Gamma,W_\Gamma]$.
\item $\pi_i(\mathbb{R}DJ_K) \cong \pi_i(\mathbb{R}\mathcal{Z}_K)$ for $i\geqslant 2$. 
\item Both spaces $\mathbb{R}DJ_K$ and $\mathbb{R}\mathcal{Z}_K$ are aspherical if and only if $K$ is flag.
\end{enumerate}
\end{theorem}

We therefore have $\mathbb{R}DJ_K\simeq BW_\Gamma$ and $\mathbb{R}\mathcal{Z}_K\simeq B[W_\Gamma,W_\Gamma]$ when $K$ is flag. For all $K$, the map $\R DJ_K \to B\Z_2^m$ of \eqref{main_fib} induces the abelianization map $\mathrm{ab}\colon W_\Gamma \to \Z_2^m$ on fundamental groups with kernel the commutator subgroup $\pi_1(\R\mathcal{Z}_K)\cong [W_\Gamma,W_\Gamma]$.

For a subset $I\subseteq [m]$, the \emph{standard parabolic subgroup} $W_{\Gamma_I}$ of a right-angled Coxeter group $W_\Gamma$ is the subgroup generated by those generators $g_i$ with $i\in I$. As our notation suggests, each standard parabolic is itself a right-angled Coxeter group corresponding to an induced subgraph~$\Gamma_I$. Moreover, each standard parabolic subgroup is a retract of $W_\Gamma$. This follows from a basic property of the real Davis--Januszkiewicz space which we record below. 

For a simplicial complex $K$ on~$[m]$, let $K_I$ denote the \emph{full subcomplex} $K_I=\{\sigma \in K : \sigma \subseteq I\}$.

\begin{lemma} \label{lem_retract}
Let $K$ be a simplicial complex on the vertex set $[m]$ with underlying graph $\Gamma$, and let $I\subseteq [m]$.
\begin{enumerate}[label={\normalfont(\alph*)}]
\item The natural inclusion $\R DJ_{K_I} \hookrightarrow \R DJ_K$ admits a retraction $\pi_I\colon \R DJ_K \to \R DJ_{K_I}$. \label{lem_retract_a}
\item The subgroup inclusion $W_{\Gamma_I} \hookrightarrow W_\Gamma$ admits a retraction $\phi_I\colon W_\Gamma \to W_{\Gamma_I}$. \label{lem_retract_b}
\end{enumerate}
\end{lemma}

\begin{proof}
Recall that, by definition, $\R DJ_K=(\R P^\infty,\, pt)^K$ is a subspace of the product $(\R P^\infty)^m$. It is straightforward to check that the projection $(\R P^\infty)^m \to (\R P^\infty)^I=\prod_{i\in I} \R P^\infty$ restricts to a map $\pi_I\colon \R DJ_K \to \R DJ_{K_I}$ which is the identity on $\R DJ_{K_I}$, which proves part~\ref{lem_retract_a}. Part~\ref{lem_retract_b} follows from \cref{thm_aspherical}\ref{thm_aspherical_a} by taking $\phi_I=(\pi_I)_*$ to be the induced map of fundamental groups.
\end{proof}

\subsection{Classifying spaces of coabelian subgroups}

Let $G$ be a group. A subgroup $H\leqslant G$ is called a \emph{coabelian subgroup} of $G$ if it contains the commutator subgroup $[G,G]$ (or, equivalently, $H$ is a normal subgroup and $G/H$ is abelian).

Let $K$ be a simplicial complex on the vertex set $[m]$ with underlying graph $\Gamma=K^{(1)}$. Each surjective group homomorphism $\rho\colon \Z_2^m\to \Z_2^n$ gives rise to:
\begin{enumerate}
\item a coabelian subgroup $G_\rho=\ker(\rho\circ\mathrm{ab})$ of the right-angled Coxeter group $W_\Gamma$; and
\item a homotopy quotient of $\R\mathcal{Z}_K$ with fundamental group $G_\rho$. \label{enum_item_2}
\end{enumerate}
Here $\mathrm{ab}\colon W_\Gamma \to \Z_2^m$ is the abelianization map with kernel $[W_\Gamma,W_\Gamma]$, so $G_\rho$ is clearly coabelian. To see \eqref{enum_item_2}, consider the homotopy quotient by the restricted action of $\ker\rho \cong \Z_2^{m-n}$ on $\R\mathcal{Z}_K$. The Borel fibration of this action is related to the Borel fibration \eqref{main_fib} of the entire $\Z_2^m$-action by a homotopy fibration diagram
\begin{equation} \label{fib_diagram}
    \begin{tikzcd}
    \Z_2^{m-n} \ar[r]\ar[d,transform canvas={xshift=0.35ex},-]\ar[d,transform canvas={xshift=-0.35ex},-] & \Z_2^m \ar[r,"\rho"] \ar[d] & \Z_2^n \ar[r] \ar[d] & B\Z_2^{m-n} \ar[d,transform canvas={xshift=0.35ex},-]\ar[d,transform canvas={xshift=-0.35ex},-] \\
    \Z_2^{m-n} \ar[r] & \R\mathcal{Z}_K \ar[r]\ar[d] & (\R\mathcal{Z}_K)_{\Z_2^{m-n}} \ar[r]\ar[d] & B\Z_2^{m-n} \\
     & \R DJ_K \ar[r,transform canvas={yshift=0.35ex},-]\ar[r,transform canvas={yshift=-0.35ex},-]\ar[d] & \mathbb{R}DJ_K \ar[d] & \\
     & B\Z_2^m \ar[r,"B\rho"] & B\Z_2^n, &
    \end{tikzcd}
\end{equation}
which exhibits the homotopy quotient $(\R\mathcal{Z}_K)_{\Z_2^{m-n}}$ as the homotopy fibre of the composite map $\R DJ_K \to B\Z_2^m \xrightarrow{B\rho} B\Z_2^n$. Since $\pi_1(\mathbb{R}DJ_K) \cong W_\Gamma$ and $\pi_1(\R\mathcal{Z}_K) \cong [W_\Gamma,W_\Gamma]$ by \cref{thm_aspherical}, applying $\pi_1(-)$ to the bottom squares of the above diagram yields
\begin{equation} \label{two_diagrams}
    \begin{tikzcd}
    \R\mathcal{Z}_K \ar[r] \ar[d] & (\R\mathcal{Z}_K)_{\Z_2^{m-n}} \ar[d] \\
    \mathbb{R}DJ_K \ar[r,transform canvas={yshift=0.35ex},-] \ar[r,transform canvas={yshift=-0.35ex},-] \ar[d] & \mathbb{R}DJ_K \ar[d] \\
    B\Z_2^m \ar[r,"B\rho"] & B\Z_2^n
    \end{tikzcd}
    \begin{tikzcd}[column sep=small]
    {\quad} \ar[r, dotted, bend left=60, "\pi_1"] & {\quad} \ar[l, dotted, bend left=60, "B"]
    \end{tikzcd}
    \begin{tikzcd}
    {[W_\Gamma,W_\Gamma]} \ar[r] \ar[d] & G_\rho \ar[d] \\
    W_\Gamma \ar[r,transform canvas={yshift=0.35ex},-] \ar[r,transform canvas={yshift=-0.35ex},-] \ar[d,"\mathrm{ab}"] & W_\Gamma \ar[d] \\
    \Z_2^m \ar[r,"\rho"] & \Z_2^n,
    \end{tikzcd}
\end{equation}
where both columns on the right are short exact sequences of groups, and $\pi_1\big((\R\mathcal{Z}_K)_{\Z_2^{m-n}})\big)=G_\rho$, as desired.
Moreover, if $K$ is a flag complex, then \cref{thm_aspherical} and the long exact sequence of homotopy groups induced by either of the fibrations involving $(\R\mathcal{Z}_K)_{\Z_2^{m-n}}$ above show that $(\R\mathcal{Z}_K)_{\Z_2^{m-n}}$ is aspherical and hence an Eilenberg--Mac Lane space $K(G_\rho,1)$, which is the classifying space $BG_\rho$ as $G_\rho$ is discrete. In this case, the classifying space functor $B$ transforms the diagram of groups on the right to the diagram of spaces on the left, up to homotopy equivalence.

The next proposition describes a one-to-one correspondence between classifying spaces of coabelian subgroups of right-angled Coxeter groups and homotopy quotients of real moment-angle complexes over flag complexes.

\begin{proposition} \label{prop_corresp}
Let $K$ be a flag complex on the vertex set $[m]$ with underlying graph $\Gamma$. For every subgroup $A \leqslant \Z_2^m$, the homotopy quotient $(\R\mathcal{Z}_K)_A$ is the classifying space $BG$ of a coabelian subgroup $G$ of the right-angled Coxeter group $W_\Gamma$. Moreover, every coabelian subgroup of $W_\Gamma$ admits a classifying space of the form $(\R\mathcal{Z}_K)_A$ for some subgroup $A \leqslant \Z_2^m$.
\end{proposition}

\begin{proof}
Given a subgroup $A\leqslant \Z_2^m$ of corank $n$, let $\rho\colon \Z_2^m \to \Z_2^m/A\cong \Z_2^n$ denote the quotient map and let $G$ be the kernel of the composite homomorphism
\[
W_\Gamma \stackrel{\mathrm{ab}\,}{\longrightarrow} \Z_2^m \stackrel{\rho\,}{\longrightarrow} \Z_2^n. 
\]
Then $G$ is a coabelian subgroup of $W_\Gamma$. Now consider the homotopy quotient $(\R\mathcal{Z}_K)_A$ and observe that diagram~\eqref{two_diagrams} identifies $(\R\mathcal{Z}_K)_A$ with $BG$.

Conversely, given a coabelian subgroup $G$ of $W_\Gamma$, let $A=G/[W_\Gamma,W_\Gamma]$ and note that this is a subgroup of $W_\Gamma/[W_\Gamma,W_\Gamma]=\Z_2^m$. Pulling back the Borel fibration \eqref{main_fib} along the map $BA\to B\Z_2^m$ induced by the subgroup inclusion yields the Borel fibration of the restricted action of $A$ on $\R\mathcal{Z}_K$:
\[
\begin{tikzcd}
\R\mathcal{Z}_K \ar[r,transform canvas={yshift=0.35ex},-]\ar[r,transform canvas={yshift=-0.35ex},-] \ar[d] & \R\mathcal{Z}_K \ar[d] & \\
(\R\mathcal{Z}_K)_A \ar[r] \ar[d] & \R DJ_K \ar[r] \ar[d] & B(\Z_2^m/A) \ar[d,transform canvas={xshift=0.35ex},-]\ar[d,transform canvas={xshift=-0.35ex},-] \\
BA \ar[r] & B\Z_2^m \ar[r,"B\rho"] & B(\Z_2^m/A).
\end{tikzcd}
\]
Here the map labelled $B\rho$ is induced by the quotient $\rho\colon \Z_2^m\to \Z_2^m/A$, and the horizontal map above it is the evident composition. Since $K$ is flag, \cref{thm_aspherical} and the long exact sequence of homotopy groups induced by the fibration along the second row above show that $(\R\mathcal{Z}_K)_A \simeq BG$, as $\pi_1\big((\R\mathcal{Z}_K)_A)\big) = \ker(\rho\circ\mathrm{ab}) = G$. (Alternatively, the homotopy quotient $(\R\mathcal{Z}_K)_A$ can be identified with $BG$ by inspection of diagrams~\eqref{fib_diagram} and~\eqref{two_diagrams}.)
\end{proof}

\begin{example}
Under the assumptions of \cref{prop_corresp}, the homotopy quotients of the $\Z_2^m$-space $\R\mathcal{Z}_K$ by the trivial subgroup $A=\{1\}$ and by the entire $2$-torus $A=\Z_2^m$ correspond to classifying spaces of the commutator subgroup $[W_\Gamma,W_\Gamma]$ and the entire right-angled Coxeter group $W_\Gamma$, respectively, by \cref{thm_aspherical} (since $(\R\mathcal{Z}_K)_{\Z_2^m} \simeq \R DJ_K$).
\end{example}

\begin{example}
The \emph{alternating subgroup} $W_\Gamma^+$ is the index two subgroup of $W_\Gamma$ defined by the kernel of the sign character $\chi\colon W_\Gamma \to \Z_2=\{\pm 1\}$ with $\chi(g_i)=-1$ for all $i\in [m]$; see~\cite{BRR}. ($W_\Gamma^+$ is the analogue of the \emph{Bestvina-Brady group} given by the kernel of the analogously defined character $A_\Gamma \to \Z$ of the right-angled Artin group $A_\Gamma$.) Let $K$ be the flag complex with underlying graph~$\Gamma$, and let $A\leqslant \Z_2^m$ be the subgroup consisting of all $m$-tuples $(x_1,\ldots,x_m)$ with an even number of non-identity entries $x_i=-1$. Then the homotopy quotient $(\R\mathcal{Z}_K)_A$ is a classifying space for $W_\Gamma^+$.
\end{example}

\begin{remark}
Classifying spaces of Bestvina--Brady groups, Artin kernels, and other coabelian subgroups of right-angled Artin groups (or other graph product groups) can similarly be constructed as homotopy quotients of the polyhedral product $(\R,\Z)^K$ (or $(EG,G)^K$, more generally) in place of $\R\mathcal{Z}_K=(D^1,S^0)^K\simeq (E\Z_2,\Z_2)^K$ (cf.~\cite{D,GILV,LS}).
\end{remark}

Of particular importance will be the homotopy quotients of $\mathbb{R}\mathcal{Z}_K$ by \emph{coordinate subgroups} $\Z_2^I$ of $\Z_2^m$, namely those of the form
\begin{equation} \label{coord_subgp}
\Z_2^I = \big\{ (x_1,\ldots,x_m)\in \Z_2^m \,:\, x_i=1 \text{ for } i\notin I \big\}
\end{equation}
for some $I\subseteq [m]$. The next lemma describes the family of subgroups of $W_\Gamma$ which these give rise to under the correspondence of \cref{prop_corresp}.

For $J\subseteq [m]$, let $\pi_J\colon \mathbb{R}DJ_K \to \mathbb{R}DJ_{K_J}$ be the retraction of \cref{lem_retract}\ref{lem_retract_a}, and let $F_{\pi_J}$ denote its homotopy fibre so that there is a (split) homotopy fibration
\begin{equation} \label{proj_fib}
F_{\pi_J} \longrightarrow \mathbb{R}DJ_K \stackrel{\pi_J}{\longrightarrow} \mathbb{R}DJ_{K_J}.
\end{equation}
The induced map of fundamental groups $(\pi_J)_*=\phi_J\colon W_\Gamma \to W_{\Gamma_J}$ is the natural projection onto the standard parabolic subgroup on the vertices $J\subseteq [m]$ (see \cref{lem_retract}).

\begin{lemma} \label{lem_coord_quot}
Let $K$ be a flag complex on the vertex set $[m]$ with underlying graph $\Gamma$. Let $I\subseteq [m]$ and $J=[m]\setminus I$.
\begin{enumerate}[label={\normalfont(\alph*)}]
\item $F_{\pi_J}\simeq B\overline{W}_{\Gamma_I}$, where $\overline{W}_{\Gamma_I}$ is the normal closure of $W_{\Gamma_I}$ in $W_\Gamma$. \label{part_a} 
\item $(\R\mathcal{Z}_K)_{\Z_2^I} \simeq BG$, where $G$ is the subgroup of $W_\Gamma$ given by the internal semidirect product
\[
G=\overline{W}_{\Gamma_I} \rtimes [W_{\Gamma_J},W_{\Gamma_J}].
\] \label{part_b}
\end{enumerate}
\end{lemma}

\begin{proof}
Since $K$ and its full subcomplex $K_J$ are flag, we have equivalences $\mathbb{R}DJ_K\simeq BW_\Gamma$ and $\mathbb{R}DJ_{K_J}\simeq BW_{\Gamma_J}$ by \cref{thm_aspherical}, so it follows immediately from the homotopy fibration~\eqref{proj_fib} that $F_{\pi_J}$ is the classifying space of the subgroup $\ker(\pi_J)_*=\ker\phi_J$. Clearly $\ker\phi_J$ is a normal subgroup of $W_\Gamma$ containing $W_{\Gamma_I}$; that it is the smallest such subgroup $\overline{W}_{\Gamma_I}$ is~\cite[Proposition~2.1(b)]{G}.

Next, observe that $(\mathbb{R}\mathcal{Z}_K)_{(\mathbb{Z}_2)^I}$ occurs as the homotopy quotient in diagram~\eqref{fib_diagram} when the coordinate homomorphism $\rho$ is taken to be the projection $\Z_2^m \to \Z_2^J$ with $J=[m]\setminus I$. It therefore follows from diagram~\eqref{two_diagrams} that
\[
\pi_1\big((\R\mathcal{Z}_K)_{\Z_2^I}\big) = \ker \big( W_\Gamma \xrightarrow{\mathrm{ab}} \Z_2^m \xrightarrow{\mathrm{proj}} \Z_2^J \big).
\]
Letting $G$ denote this kernel, we have $(\R\mathcal{Z}_K)_{\Z_2^I} \simeq BG$ by the discussion following~\eqref{two_diagrams} since $K$ is flag. To identify $G$ more precisely, note that the group homomorphism above equals the composite $W_\Gamma \xrightarrow{\phi_J} W_{\Gamma_J} \xrightarrow{\mathrm{ab}} \Z_2^J$. Taking kernels vertically in the bottom square below therefore leads to a commutative diagram
\[
\begin{tikzcd}
\overline{W}_{\Gamma_I} \ar[r] \ar[d,transform canvas={xshift=0.35ex},-]\ar[d,transform canvas={xshift=-0.35ex},-] & G \ar[r] \ar[d] & {[W_{\Gamma_J},W_{\Gamma_J}]} \ar[d] \ar[l, dashrightarrow, bend right=33] \\
\overline{W}_{\Gamma_I} \ar[r] & W_\Gamma \ar[r,"\phi_J"] \ar[d] & W_{\Gamma_J} \ar[d, "\mathrm{ab}"] \ar[l, dashrightarrow, bend right=30] \\
 & \Z_2^J \ar[r,transform canvas={yshift=0.35ex},-]\ar[r,transform canvas={yshift=-0.35ex},-] & \Z_2^J
\end{tikzcd}
\]
where each row and column is a short exact sequence of groups, and the sections represented by dotted arrows are subgroup inclusions. The split short exact sequence along the top row identifies $G=\pi_1\big((\R\mathcal{Z}_K)_{\Z_2^I}\big)$ with the desired semidirect product.
\end{proof}

\begin{remark}
We note that the normal closure $\overline{W}_{\Gamma_I} \cong \pi_1(F_{\pi_J})$ of a standard parabolic subgroup appearing in \cref{lem_coord_quot} is itself a Coxeter group; see \cite[Proposition~2.1]{G}.
\end{remark}

\section{Equivariant formality for real moment-angle complexes} \label{sec_3}

The goal of this section is to prove \cref{main_thm_2} by establishing a real analogue of the following equivariant formality criterion for ordinary moment-angle complexes from~\cite{AB}.

\begin{theorem}[{\cite[Theorem~C]{AB}}] \label{thm_AB}
Let $K$ be a flag complex on the vertex set $[m]$, and let $I\subseteq[m]$. Then the coordinate $T^I$-action on $\mathcal{Z}_K$ is equivariantly formal over a field $k$ if and only if $I\in K$ and $K_{\{i,j\}}\ast K_{I\setminus \{i,j\}} \subseteq K$ for every missing edge $\{i,j\} \notin K$.
\end{theorem}

Rather than redoing the work of~\cite{AB} in the context of elementary abelian $2$-group actions on real moment-angle complexes $\R\mathcal{Z}_K$, instead of torus actions on $\mathcal{Z}_K$, we identify the equivariant formality of coordinate subgroup actions on $\R\mathcal{Z}_K$ and $\mathcal{Z}_K$ by describing their fixed points and applying the criteria of \cref{thm_bettisumcrit,thm_compatiblecrit}.

\subsection{Reduction to coordinate subgroup actions}
To classify the subgroups $A\leqslant \Z_2^m$ that act equivariantly formally on $\R\mathcal{Z}_K$ for a given simplicial complex $K$, we first show that the problem can be reduced to the case of coordinate subgroups $A=\Z_2^I$ (see~\eqref{coord_subgp}).

\begin{definition} \label{def_hull}
Let $A$ be a subgroup of $\Z_2^m$. The \emph{coordinate hull} of $A$, denoted $\operatorname{hull}(A)$, is the smallest coordinate subgroup of $\Z_2^m$ containing $A$. More precisely,
\[
\operatorname{hull}(A) = \Z_2^I \quad\text{for}\quad 
I=\big\{ i\in[m] \,:\, \mathrm{proj}_i(A)\neq\{1\} \big\},
\]
where $\mathrm{proj}_i\colon \Z_2^m \to \Z_2$ is the $i^\text{th}$ coordinate projection.
\end{definition}

\begin{lemma} \label{lem_fixedpts}
Let $K$ be a simplicial complex on the vertex set $[m]$, and let $A$ be a subgroup of $\Z_2^m$. Then there is an equality of fixed points
\[
(\R\mathcal{Z}_K)^A = (\R\mathcal{Z}_K)^{\operatorname{hull}(A)}.
\]
\end{lemma}

\begin{proof}
Clearly $(\R\mathcal{Z}_K)^{\operatorname{hull}(A)} \subseteq (\R\mathcal{Z}_K)^A$ since $A\subseteq \operatorname{hull}(A)$. Conversely, suppose $(x_1,\ldots,x_m)\in \R\mathcal{Z}_K\subseteq (D^1)^m$ is fixed by every $a\in A$. Then $\mathrm{proj}_i(a)\cdot x_i=x_i$ for all $i\in [m]$ and $a\in A$, since the $A$-action on $\R\mathcal{Z}_K$ is a restriction of the coordinatewise $\Z_2^m$-action on $(D^1)^m$. It follows that $(x_1,\ldots,x_m)$ is fixed by each coordinate generator of $\operatorname{hull}(A)$, and hence lies in $(\R\mathcal{Z}_K)^{\operatorname{hull}(A)}$.
\end{proof}

The following is a real analogue of~\cite[Proposition~5.6]{AB}.

\begin{proposition} \label{prop_ef_hull}
Let $K$ be a simplicial complex on the vertex set $[m]$, and let $A$ be a subgroup of~$\Z_2^m$. Then $\R\mathcal{Z}_K$ is an equivariantly formal $A$-space over $\mathbb{F}_2$ if and only if $\R\mathcal{Z}_K$ is an equivariantly formal $\operatorname{hull}(A)$-space over $\mathbb{F}_2$.
\end{proposition}

\begin{proof}
Since $(\R\mathcal{Z}_K)^A = (\R\mathcal{Z}_K)^{\operatorname{hull}(A)}$ by \cref{lem_fixedpts}, the proposition follows from the Betti sum criterion for equivariant formality in \cref{thm_bettisumcrit}.
\end{proof}

\subsection{Detecting equivariant formality via fixed points}

We recall that for a simplicial complex $K$ on the vertex set $[m]$, the \emph{link} of a face $\sigma \in K$ is the subcomplex
\[
\lk_K(\sigma) = \{ \tau \in K \,:\, \tau \cap \sigma = \varnothing \text{ and } \tau \cup \sigma \in K  \}
\]
on the vertex set $[m] \setminus \sigma$.

Recall also that the ordinary moment-angle complex $\mathcal{Z}_K \subseteq (D^2)^m$ comes with a standard action of the $m$-torus $T^m$. The diagonal $\Z_2$-action on the polydisc $(D^2)^m$ defined by complex conjugation also restricts to an involution on $\mathcal{Z}_K$.

\begin{lemma} \label{lem_coordfixedpts}
Let $I \subseteq [m]$. Then the fixed point set $(\mathcal{Z}_K)^{T^I}$ is $\Z_2$-equivariantly homeomorphic to $\mathcal{Z}_{\lk_K(I)}$ with respect to the complex conjugation action if $I \in K$, and is empty otherwise.
\end{lemma}

\begin{proof}
First note that $(x_1, \ldots, x_m) \in (D^2)^m$ is fixed by $T^I$ if and only if $x_i = 0$ for all $i \in I$. Let $J=[m] \setminus I$. Then the fixed point set of the $T^I$-action on $(D^2)^m$ is $(D^2,0)^{J}$. Hence,
\[
(\mathcal{Z}_K)^{T^I}=\mathcal{Z}_K\cap (D^2,0)^J = \bigcup_{\sigma\in K} \big( (D^2,S^1)^{\sigma} \cap (D^2,0)^J \big). 
\]
The intersection $(D^2,S^1)^{\sigma} \cap (D^2,0)^J$ is nonempty if and only if $I \subseteq \sigma$. Therefore, the fixed point set $(\mathcal{Z}_K)^{T^I}$ is nonempty if and only if $I \in K$.  For $I \in K$, we have 
\begin{equation} \label{fixedptsunion}
(\mathcal{Z}_K)^{T^I} = \bigcup_{\sigma\in K,\, I \subseteq \sigma} \big( (D^2,S^1)^{\sigma} \cap (D^2, 0)^J \big).
\end{equation}
For $\sigma\in K$ with $I \subseteq \sigma$, let $J'= \sigma\setminus I$ so that $\sigma = J' \sqcup I$. Observe that $J ' \in \lk_K(I)$, and every face of $\lk_K(I)$ arises in this way from some $\sigma\in K$ containing $I$. We have
\[
(D^2,S^1)^{\sigma} \cap (D^2,0)^J = \prod_{i \in J'} D^2 \times  \prod_{i \in I} \{0\} \times \prod_{i \notin \sigma} S^1 \cong \prod_{i \in J'}D^2 \times \prod_{i \notin \sigma}S^1,
\]
and since~\eqref{fixedptsunion} can be written as a union over all faces of the link $\lk_K(I)$, this gives the desired $\Z_2$-equivariant homeomorphism $(\mathcal{Z}_K)^{T^I} \cong \mathcal{Z}_{\lk_K(I)}$.
\end{proof}

\begin{theorem} \label{thm_ef}
Let $K$ be a simplicial complex on the vertex set $[m]$, and let $I \subseteq [m]$. Then the following conditions are equivalent:
\begin{enumerate}[label={\normalfont(\alph*)}]
\item \label{ef_a} $\R\mathcal{Z}_K$ is an equivariantly formal $\Z_2^I$-space over $\mathbb{F}_2$;
\item \label{ef_b} $\mathcal{Z}_K$ is an equivariantly formal $T^I$-space over $\mathbb{F}_2$;
\item \label{ef_c} $K_J\setminus(I\cap J) \hookrightarrow K_J$ induces the trivial map on $\widetilde{H}^*(-;\F_2)$ for all $J\subseteq [m]$.
\end{enumerate}
Moreover, if $K$ is flag, then the conditions above are equivalent to:
\begin{enumerate}[resume,label={\normalfont(\alph*)}]
\item \label{ef_d} $I\in K$ and $\{i,j_1\}, \{i,j_2\} \in K$ for every missing edge $\{j_1,j_2\} \notin K$ and every $i\in I\setminus\{j_1,j_2\}$.
\end{enumerate}
\end{theorem}

\begin{proof}
The equivalence of conditions~\ref{ef_b} and~\ref{ef_c} is contained in~\cite[Theorem~B]{AB}, and the equivalence of~\ref{ef_b} and~\ref{ef_d} when $K$ is flag is \cref{thm_AB}. (Note that since a flag complex is determined by its underlying graph,~\ref{ef_d} is equivalent to the condition that $I\in K$ and $K_{\{j_1,j_2\}}\ast K_{I\setminus \{j_1,j_2\}} \subseteq K$ for every $\{j_1,j_2\} \notin K$.) It therefore suffices to prove the equivalence of~\ref{ef_a} and~\ref{ef_b}.
 
Consider the complex conjugation $\Z_2$-action on $\mathcal{Z}_K$. It is easily checked that $(\mathcal{Z}_K)^{\Z_2} = \R\mathcal{Z}_K$. Moreover, it follows from the well-known additive isomorphisms~\cite[Corollary~8.3.6]{BP}
\[
H^n(\R\mathcal{Z}_K;\Z)\cong \bigoplus_{J\subseteq[m]} \widetilde{H}^{n-1}(K_J;\Z), \qquad H^n(\mathcal{Z}_K;\Z)\cong \bigoplus_{J\subseteq[m]} \widetilde{H}^{n-|J|-1}(K_J;\Z)
\]
that $\dim_{\F_2}H^*(\mathcal{Z}_K;\F_2) = \dim_{\F_2}H^*(\R\mathcal{Z}_K;\F_2)$. This, together with the fixed point criterion of \cref{thm_bettisumcrit}, implies that the $\Z_2$-action on $\mathcal{Z}_K$ is equivariantly formal. On the other hand, by \cref{lem_coordfixedpts}, we have that $(\mathcal{Z}_K)^{T^I} \cong \mathcal{Z}_{\lk_K(I)}$ as $\Z_2$-spaces if $I\in K$. So in this case the complex conjugation $\Z_2$-action on $(\mathcal{Z}_K)^{T^I}$ is equivariantly formal by the same argument as before, and \cref{thm_compatiblecrit} implies that $\R\mathcal{Z}_K$ is an equivariantly formal $\Z_2^I$-space over $\F_2$ if and only if $\mathcal{Z}_K$ is an equivariantly formal $T^I$-space over $\F_2$. Finally, if $I\notin K$, then the actions by the coordinate subgroups $\Z_2^I$ and $T^I$ on $\R\mathcal{Z}_K$ and $\mathcal{Z}_K$, respectively, have no fixed points. Therefore, neither action is equivariantly formal, making conditions~\ref{ef_a} and~\ref{ef_b} equivalent in this case as well.
\end{proof}

\begin{proof}[Proof of \cref{main_thm_2}]
Conditions~\ref{main_thm_2a} and~\ref{main_thm_2b} of \cref{main_thm_2} are equivalent by \cref{prop_ef_hull}, while the equivalence of~\ref{main_thm_2b} and~\ref{main_thm_2c} follows from \cref{thm_ef}.
\end{proof}

\section{Proof of Theorem 1.1} \label{sec_4}

In this short section we derive \cref{main_thm_1} from \cref{main_thm_2} and discuss some consequences.

\begin{proof}[Proof of \cref{main_thm_1}]
Let $G$ be a coabelian subgroup of a right-angled Coxeter group $W_\Gamma$. The elementary abelian $2$-group $A=G/[W_\Gamma,W_\Gamma]$, viewed as a subgroup of $W_\Gamma/[W_\Gamma,W_\Gamma]=\Z_2^m$, has a coordinate hull of the form $\operatorname{hull}(A)=\Z_2^I$ for some $I\subseteq [m]$.

We associate to $G$ another subgroup of $W_\Gamma$, namely the internal semidirect product
\[ G^\rtimes = \overline{W}_{\Gamma_I} \rtimes [W_{\Gamma_J},W_{\Gamma_J}], \]
where $\overline{W}_{\Gamma_I}$ is the normal closure of the standard parabolic subgroup $W_{\Gamma_I}$, and $J=[m]\setminus I$. Now let $K$ be the clique complex of the graph $\Gamma$. Then $(\R\mathcal{Z}_K)_A$ and $(\R\mathcal{Z}_K)_{\Z_2^I}$ are classifying spaces of $G$ and $G^\rtimes$, respectively, by \cref{prop_corresp} and \cref{lem_coord_quot}. Moreover, the Borel fibrations
\[
\R\mathcal{Z}_K \to (\R\mathcal{Z}_K)_A \to BA, \qquad \R\mathcal{Z}_K \to (\R\mathcal{Z}_K)_{\Z_2^I} \to B\Z_2^I
\]
can be identified, up to homotopy equivalence, with the homotopy fibrations obtained by applying the classifying space functor $B$ to the short exact sequences of groups
\[ 
1\to [W_\Gamma,W_\Gamma]\to G\to A\to 1, \qquad 
1\to [W_\Gamma,W_\Gamma]\to G^\rtimes\to \Z_2^I\to 1. 
\]
Since the set of homotopy classes of pointed maps $[BG,BA]$ (resp., $[BG^\rtimes,B\Z_2^I]$) is in bijection with $\operatorname{Hom}(G,A)$ (resp., $\operatorname{Hom}(G^\rtimes,\Z_2^I)$), this follows from the fact that the fibrations above induce these short exact sequences on fundamental groups, as can be seen from~\eqref{two_diagrams}, for example.

In particular, $H_A^*(\R\mathcal{Z}_K;\F_2)$ is isomorphic to the group cohomology $\mathrm{H}^*(G;\F_2)$ as a graded module over the polynomial ring $\mathrm{H}^*(A;\F_2)=H^*(BA;\F_2)$. Similarly, $H_{\Z_2^I}^*(\R\mathcal{Z}_K;\F_2) \cong \mathrm{H}^*(G^\rtimes;\F_2)$ as graded modules over $\mathrm{H}^*(\Z_2^I;\F_2) \cong \F_2[x_i : i\in I]$. Under these correspondences, the equivalence of conditions~\ref{main_thm_1a},~\ref{main_thm_1b}, and~\ref{main_thm_1c} in \cref{main_thm_1} now follows from the equivalence of the corresponding conditions in \cref{main_thm_2}.
\end{proof}

The cohomology of a group is called \emph{Cohen--Macaulay} of dimension $r$ if $\mathrm{H}^*(G;k)$ is finitely generated and free over a polynomial subalgebra $k[x_1,\ldots,x_r] \subseteq \mathrm{H}^*(G;k)$.

\begin{corollary}
Let $\Gamma$ be a simple graph on $m$ vertices, and let $G$ be a coabelian subgroup of $W_\Gamma$. If the equivalent conditions of \cref{main_thm_1} hold, then $\mathrm{H}^*(G;\F_2)$ is Cohen--Macaulay of dimension $m-\operatorname{corank}(G)$.
\end{corollary}

Under the assumptions above, the rank of the free module $\mathrm{H}^*(G;\F_2)$ over the polynomial ring $\mathrm{H}^*(G/[W_\Gamma,W_\Gamma];\F_2)$ can be described in terms of the well-known cohomology of $[W_\Gamma,W_\Gamma]$:

\begin{corollary}
Let $G$ be a coabelian subgroup of $W_\Gamma$, and let $r=\operatorname{rank}(G/[W_\Gamma,W_\Gamma])$. If the equivalent conditions of \cref{main_thm_1} hold, then there is an isomorphism of $\F_2[x_1,\ldots,x_r]$-modules
\[
\mathrm{H}^*(G;\F_2)\cong \Biggl( \bigoplus_{J\subseteq[m]} \widetilde{H}^{*-1}(K_J;\F_2) \Biggr) \otimes_{\F_2} \F_2[x_1,\ldots,x_r],
\]
where $K$ is the flag complex of $\Gamma$, and $K_J$ is the full subcomplex of $K$ on the vertex set $J\subseteq[m]$.
\end{corollary}

\begin{proof}
This follows from the collapse of the Lyndon--Hochschild--Serre spectral sequence of the short exact sequence $1\to [W_\Gamma,W_\Gamma]\to G\to G/[W_\Gamma,W_\Gamma]\to 1$ (see also \cref{thm_bettisumcrit}) and the Hochster decomposition $\mathrm{H}^n([W_\Gamma,W_\Gamma];\F_2)\cong H^n(\R\mathcal{Z}_K;\F_2)\cong \bigoplus_{J\subseteq[m]} \widetilde{H}^{n-1}(K_J;\F_2)$ of~\cite[Corollary~8.3.6]{BP}.
\end{proof}

\appendix
\section{Betti number criteria for equivariant formality} \label{app_A}

Here we collect various general criteria for equivariant formality over $\F_2$ (\cref{thm_bettisumcrit}) and prove a general result (\cref{thm_compatiblecrit}) on compatible group actions that allows us to relate the equivariant formality of certain torus and elementary abelian $2$-group (or ``$2$-torus'') actions over~$\F_2$. The arguments are standard and follow~\cite{AP}.

We use the same assumptions on topological spaces as in~\cite{AFP}: all spaces $X$ are assumed to be compact, Hausdorff, second-countable, and locally contractible. Examples include topological manifolds, complex algebraic varieties, and countable, locally finite CW-complexes. Furthermore, we let $T=(S^1)^r$ and assume that \emph{all stabilizers $T_x$ are connected} for any $T$-space $X$.

Let $X$ be a $T$-space and let $k$ be a field. Let $R=H^*(BT;k)$ and let $p^*\colon R \to H^*_T(X;k)$ be the map induced by the projection map $p\colon X_T \to BT$. For each orbit $Tx$, we have $(Tx)_T \simeq BT_x$, and hence the inclusion $j_x\colon Tx \hookrightarrow X$ induces a map $j_x^* \colon H^*_T(X;k) \to H^*(BT_x;k)$. For each point $x \in X$, we thus have a map $j_x^*\circ p^* \colon R \to H^*(BT_x;k)$. 
For a multiplicative set $S \subseteq R$, define 
\[ 
X^S = \big\{ x \in X \,:\, (j_x^*\circ p^*)(s) \neq 0 \text{ for all } s \in S \big\}.
\]
With these conventions, we can now state the localization theorem.

\begin{theorem}[{Localization Theorem~\cite[Theorem~3.2.6]{AP}}] \label{thm_localization}
Let $k=\mathbb{Q}$ or $\F_p$, where $p$ is prime. Let $\phi^*\colon H^*_T(X;k) \to H_T^*(X^S;k)$ be the map induced by the inclusion $\phi\colon X^S \hookrightarrow X$. Then the induced localized homomorphism
\[
S^{-1}\phi^* \colon S^{-1} H_T^*(X;k) \longrightarrow S^{-1} H_T^*(X^S;k)
\]
is an isomorphism of $S^{-1}R$-algebras.
\end{theorem}

The next lemma is the key point, where we use that all stabilizers are connected. 

\begin{lemma} \label{S-localization}
    Let $S = R \setminus \{0\}$ be the multiplicative set of $R$, and let $X^T$ be the fixed points of the $T$-space $X$. Then $X^S = X^T$. 
\end{lemma}

\begin{proof}
    Let $x \in X^T$. Then $j_x^*\circ p^* = \mathrm{id}$ and hence $x \in X^S$. Conversely, suppose $x \notin X^T$. Since all stabilizers are connected, $T_x$ is a subtorus of $T$ and $T_x \neq T$, hence $\rk T_x < \rk  T$. Therefore,  $\ker(j_x^*\circ p^*) \neq 0 $, but this implies $x \notin X^S$ since $S = R \setminus \{0\}$.   
\end{proof}

\begin{proposition} \label{SScollapse}
    Let $X$ be a $T$-space with $\dim_{\mathbb{F}_p}H^*(X; \mathbb{F}_p) < \infty$. Then 
    \[
    \dim_{\mathbb{F}_p}H^*(X^T;\mathbb{F}_p) \leqslant \dim_{\mathbb{F}_p}H^*(X;\mathbb{F}_p).
    \]
    Moreover, the Serre spectral sequence $E_2^{i,j} = H^i(BT, H^j(X; \mathbb{F}_p)) \Rightarrow H_T^{i+j}(X;\F_p)$ of the fibration $X \to X_T \to BT$ degenerates at $E_2$ if and only if $\dim_{\F_p}H^*(X^T;\F_p) =\dim_{\F_p}H^*(X;\F_p)$.
\end{proposition}

\begin{proof}
Let $R = H^*(BT;\mathbb{F}_p)$. Since the rows of $E^{*,*}_r$ are $R$-modules and the differentials $d_r$, $r \geqslant 2$, are $R$-module homomorphisms, one can localize the spectral sequence with respect to $S = R \setminus \{0\}$. Let $K = R_{(0)}$ be the field of fractions. The spectral sequence $\{S^{-1}E_r\}_{r\geqslant 2}$ becomes a sequence of complexes of
vector spaces over $K$ with $S^{-1}E_2 \cong K \otimes_{\mathbb{F}_p} H^*(X;\mathbb{F}_p)$. Now we have 
\[
\begin{aligned}
 \dim_{\mathbb{F}_p}H^*(X;\mathbb{F}_p)
  &= \dim_{K}(S^{-1}E_2)
   \geqslant \dim_{K}(S^{-1}E_3) \geqslant\cdots \\
  &\geqslant \dim_{K}(S^{-1}E_\infty)
   = \dim_{K}\bigl(S^{-1}H_T^*(X;\mathbb{F}_p)\bigr) \\
  &= \dim_{K}\bigl(S^{-1}H_T^*(X^T;\mathbb{F}_p)\bigr)
   = \dim_{\mathbb{F}_p}H^*(X^T;\mathbb{F}_p),
\end{aligned}
\]
where the equality $\dim_{K}\bigl(S^{-1}H_T^*(X;\mathbb{F}_p)\bigr) = \dim_{K}\bigl(S^{-1}H_T^*(X^T;\mathbb{F}_p)\bigr)$ follows from \cref{thm_localization} and \cref{S-localization}. Therefore, $\dim_{\mathbb{F}_p}H^*(X^T;\mathbb{F}_p) \leqslant \dim_{\mathbb{F}_p}H^*(X; \mathbb{F}_p)$ with equality if and only if $E_2 = E_\infty$.
\end{proof}

We next recall a general criterion for equivariant formality which generalizes the Smith--Thom inequality for spaces with an involution. For a $G$-space $X$, let $i^*\colon H^*_G(X;\F_2) \to H^*(X;\F_2)$ denote the forgetful map induced by the fibre inclusion of the Borel fibration $X\xrightarrow{i} X_G\to BG$.

\begin{theorem} \label{thm_bettisumcrit}
Let $G=T^r$ or $\mathbb{Z}_2^r$. Let $X$ be a $G$-space with the same assumptions as in the beginning of this section. 
Then
\[
\dim_{\mathbb{F}_2}H^*(X^G;\mathbb{F}_2) \leqslant \dim_{\mathbb{F}_2}H^*(X;\mathbb{F}_2).
\]
Moreover, $\dim_{\F_2}H^*(X^G;\F_2) = \dim_{\F_2}H^*(X;\F_2)$ if and only if the following equivalent conditions hold:
\begin{enumerate}[label={\normalfont(\alph*)}]
    \item $i^* \colon H_G^*(X; \mathbb{F}_2) \to H^*(X;\mathbb{F}_2)$ is surjective;
    \item $H_G^*(X; \mathbb{F}_2)$ is a free module over the polynomial ring $H^*(BG; \mathbb{F}_2)$;
    \item $G$ acts trivially on $H^*(X;\mathbb{F}_2)$, and the Serre spectral sequence with mod $2$ coefficients associated to the fibration $X\to X_G \to BG$ degenerates at the second page.
\end{enumerate}
\end{theorem}

\begin{proof}
See~\cite[Theorem~3.10.4]{AP} and~\cite[Theorem~10.2]{AFP} for the $G=\Z_2^r$ case, and~\cite[Theorem~8.3]{AFP} and \cref{SScollapse} for the $G=T^r$ case. 
\end{proof}

\begin{definition} \label{def_eq}
A $G$-space $X$ is called \emph{equivariantly formal over $\mathbb{F}_2$} 
if any of the equivalent conditions of Theorem~\ref{thm_bettisumcrit} hold.
\end{definition}

Suppose now that $\Z_2 = \langle \tau \rangle$ acts on a $T$-space $X$ and that this action is \emph{compatible} with the torus action in the sense that $\tau(g \cdot x) = g^{-1} \cdot \tau(x)$ for all $x \in X$ and $g \in T$. Such actions arise in algebraic and symplectic geometry, where examples include toric varieties and symplectic toric manifolds with anti-symplectic involutions; see, e.g.,~\cite{BGH}. Under this compatibility assumption, $\Z_2$ acts on $X^T$, and the maximal elementary abelian $2$-subgroup $T_2$ (consisting of all elements $g\in T$ with $g^2=1$) acts on $X^{\Z_2}$. We can now state the main result of this section.

\begin{theorem} \label{thm_compatiblecrit}
Assume that $X$ and $X^T$ are equivariantly formal as $\Z_2$-spaces over $\mathbb{F}_2$. Then $X$ is an equivariantly formal $T$-space over $\mathbb{F}_2$ if and only if $X^{\Z_2}$ is an equivariantly formal $T_2$-space over $\mathbb{F}_2$.  
\end{theorem}

\begin{proof}
Since the $\Z_2$-actions on $X$ and $X^T$ are equivariantly formal and $(X^T)^{\Z_2} = X^{T_2}$, we have $\dim_{\F_2}H^*(X^{\Z_2};\F_2) = \dim_{\F_2}H^*(X;\F_2)$ and $\dim_{\F_2}H^*(X^{T_2};\F_2) = \dim_{\F_2}H^*(X^{T};\F_2)$ by \cref{thm_bettisumcrit}. We therefore have 
\[
\begin{array}{ccc}
\dim_{\F_2}H^*(X^T;\F_2) & \leqslant & \dim_{\F_2}H^*(X;\F_2) \\
\Vert &   & \Vert \\
\dim_{\F_2}H^*(X^{T_2};\F_2) & \leqslant & \dim_{\F_2}H^*(X^{\Z_2};\F_2).
\end{array}
\]
So $\dim_{\F_2}H^*(X^{T_2};\F_2) = \dim_{\F_2}H^*(X^{\Z_2};\F_2)$ if and only if $\dim_{\F_2}H^*(X^{T};\F_2) = \dim_{\F_2}H^*(X;\F_2)$, which completes the proof, again by \cref{thm_bettisumcrit}.
\end{proof}

\end{document}